\documentclass[a4paper,12pt,spanish]{article}
\usepackage[spanish]{babel}
\usepackage{amsthm}
\usepackage{amsmath}
\usepackage{amssymb}
\usepackage{enumerate}
\usepackage[T1]{fontenc}
\usepackage{hyperref}

\newtheorem{theorem}{Teorema}[section]
\newtheorem{proposition}[theorem]{Proposici\'{o}n}
\newtheorem{corollary}[theorem]{Corolario}
\newtheorem{lemma}[theorem]{Lema}
\newtheorem*{theorem*}{Teorema}
\newtheorem*{proposition*}{Proposici\'{o}n}
\newtheorem*{corollary*}{Corolario}
\newtheorem*{lemma*}{Lema}
\theoremstyle{definition}
\newtheorem{definition}[theorem]{Definici\'{o}n}

\newcommand{\BbbN}{\mathbb{N}}

\newcommand{\BbbZ}{\mathbb{Z}}

\newcommand{\GKdim}{\mathrm{GKdim}}
\newcommand{\esc}[2]{\langle #1, #2 \rangle}
\newcommand{\pcomp}{\varphi}
\newcommand{\scomp}{\psi}
\newcommand{\mono}[2]{\mathbf{#1}^{#2}}
\newcommand{\peso}[1]{\pcomp (#1_1) + \scomp(#1_2)}
\newcommand{\filext}[1]{\mathfrak{#1}}
\newcommand{\cola}[2]{o_{\downarrow(#1,#2)}}
\newcommand{\sop}[1]{\mathrm{Sop}(#1)}
\newcommand{\graded}[2]{G^{#1}(#2)}
\newcommand{\blackhole}[3]{#1^{-}_{#2}(#3)}
\newcommand{\mdeg}[2]{\mathrm{mdeg}_{#1}(#2)}
\newcommand{\domain}{\Lambda}

\author{Jos\'{e} G\'{o}mez Torrecillas \\
Universidad de Granada}
\title{Re-filtraci\'{o}n y propiedades de regularidad de anillos
multi-filtrados}
\date{26 de septiembre de 2000}
\begin{document}
\maketitle

\begin{abstract}
Para un \'{a}lgebra $R$ multi-filtrada se muestra que, bajo
condiciones favorables sobre el \'{a}lgebra multi-graduada asociada
$G(R)$, es posible levantar condiciones homol\'{o}gicas como la
regularidad en el sentido de Auslander o la propiedad de
Cohen-Macaulay de $G(R)$ a $R$. Como aplicaci\'{o}n, obtenemos que las
$\mathbb{C}(q)$--\'{a}lgebras envolventes cuantizadas $U_q(C)$
definidas por DeConcini y Procesi en \cite{DeConcini/Procesi:1993}
son regulares Auslander y Cohen-Macaulay. Para obtener estos
resultados, se desarrolla una t\'{e}cnica de multi-filtraci\'{o}n para las
llamadas extensiones acotadas y un m\'{e}todo de re-filtraci\'{o}n, en
combinaci\'{o}n con resultados de McConnell y Stafford
(\cite{McConnell/Stafford:1989}) y DeConcini y Procesi (loc.
cit.).
\end{abstract}

\section{Propiedades de regularidad de \'{a}lgebras no conmutativas}

Comienzo recordando, para conveniencia del lector, algunas
nociones. Mis referencias fundamentales aqu\'{\i} son \cite{Bjork:1989}
y \cite{Levasseur:1992}.

Sea $R$ un anillo noetheriano y $M$ un m\'{o}dulo (por la izquierda o
por la derecha) finitamente generado sobre $R$. El \emph{n\'{u}mero
grado} de $M$ se define como
\[
j_A(M) = \inf \{\; i | \; \mathrm{Ext}_R^i(M,R) \neq 0 \} \in
\BbbN \cup \{ + \infty \}
\]

Diremos, por otra parte, que $M$ satisface la \emph{condici\'{o}n de
Auslander} si para todo $i \geqslant 0$ y todo subm\'{o}dulo $N$ de
$\mathrm{Ext}_R^i(M,R)$, se tiene que $j_R(N) \geqslant i$. Cuando
la dimensi\'{o}n global homol\'{o}gica de $R$ sea finita y todo m\'{o}dulo
finitamente generado satisfaga la condici\'{o}n de Auslander, diremos
que $R$ es \emph{regular Auslander}.

Si ahora el anillo $R$ es un \'{a}lgebra finitamente generada sobre un
cuerpo $\mathbf{k}$, podemos definir la dimensi\'{o}n de
Gelfand-Kirillov de cada m\'{o}dulo (ver \cite{Krause/Lenagan:2000}
para un extenso estudio de esta dimensi\'{o}n). El \'{a}lgebra $R$ se
llama \emph{Cohen-Macaulay} si
\[
\GKdim(M) + j_R(M) = \GKdim(R)
\]
para todo $R$--m\'{o}dulo finitamente generado.

La propiedad Cohen-Macaulay, en conjunci\'{o}n con la regularidad, es
interesante porque, en su presencia, es posible estudiar la
catenaridad de algunas \'{a}lgebras cuantizadas (ver
\cite{Goodearl/Lenagan:1996}) o la estructura del \'{u}ltimo punto de
la resoluci\'{o}n inyectiva minimal de los m\'{o}dulos regulares ${}_RR$ y
$R_R$ (ver \cite{Gomez/Jara/Merino:1996}).

 Seguidamente, voy a refrescar algunos resultados b\'{a}sicos sobre
ciertas \'{a}lgebras cuantizadas que nos ser\'{a}n de utilidad. Si $Q =
(q_{ij})$ es una matriz cuadrada de tama\~{n}o $s \times s$
multiplicativamente antisim\'{e}trica con coeficientes en
$\mathbf{k}$, el \emph{espacio cu\'{a}ntico asociado}
$\mathcal{O}_{Q}(\mathbf{k}^s) = \mathbf{k}_Q[x_1, \dots, x_s]$
est\'{a} generado como $\mathbf{k}$--\'{a}lgebra por variables $x_1,
\dots, x_s$ sujetas a las relaciones $x_jx_i = q_{ji}x_ix_j$. Un
hecho fundamental es que este \'{a}lgebra es una extensi\'{o}n iterada de
Ore
\[
\mathcal{O}_{Q}(\mathbf{k}^s) =
\mathbf{k}[x_1][x_2;\sigma_2]\cdots[x_s;\sigma_s],
\]
donde $\sigma_j(x_i) = q_{ji}x_i$ para $1 \leqslant i < j
\leqslant s$.

Para nuestros prop\'{o}sitos nos interesar\'{a}n algunas localizaciones
sencillas de los espacios cu\'{a}nticos. As\'{\i}, tomemos algunas de las
variables que, por simplicidad en la notaci\'{o}n, suponemos ser $x_1,
\dots x_t$ con $t \leqslant s$. Como $x_1, \dots, x_t$ son
elementos normales, el conjunto multiplicativo $S$ que generan es
un conjunto de Ore (ver \cite[Lemma 4.1]{Krause/Lenagan:2000}),
podemos considerar la localizaci\'{o}n de
$\mathcal{O}_{Q}(\mathbf{k}^s)$ en este conjunto, que denotaremos
por
\[
\mathbf{k}_{Q}[x_1^{\pm 1}, \dots, x_t^{\pm 1}, x_{t+1}, \dots,
x_{s}]
\]

El siguiente resultado es bien conocido en este \'{a}rea; no obstante,
no he encontrado una referencia c\'{o}moda, por lo que incluyo su
enunciado y una demostraci\'{o}n.

\begin{proposition}\label{grCM}
El \'{a}lgebra $A = \mathbf{k}_{Q}[x_1^{\pm 1}, \dots, x_t^{\pm 1},
x_{t+1}, \dots, x_{s}]$ es regular Auslander y Cohen-Macaulay.
\end{proposition}
\begin{proof}
Es claro que $A$ es una extensi\'{o}n iterada de Ore de un \'{a}lgebra de
McConnell-Pettit, luego su dimensi\'{o}n global homol\'{o}gica es finita
por \cite[3.1]{McConnell/Pettit:1988} y \cite[Theorem
4.2]{Ekstrom:1989}. Por otra parte el espacio af\'{\i}n cu\'{a}ntico
\[
\mathbf{k}_{Q}[x_1, \dots, x_t,x_{t+1}, \dots, x_s]
\]
es regular Auslander y Cohen-Macaulay (ver, por ejemplo,
\cite[Theorem 3.5]{Goodearl/Lenagan:1996}). Por \cite[Proposition
2.1]{Ajitabh/Smith/Zhang:1999}, $A$ verifica la condici\'{o}n de
Auslander. Como el conjunto multiplicativo generado por $x_1,
\dots, x_t$ est\'{a} formado por monomios que son elementos normales
locales en el espacio af\'{\i}n cu\'{a}ntico, entonces, por \cite[Theorem
2.4]{Ajitabh/Smith/Zhang:1999}, nuestra \'{a}lgebra $A$ es as\'{\i}mismo
Cohen-Macaulay.
\end{proof}

\section{\'{A}lgebras y m\'{o}dulos multi-filtrados}
En este trabajo $\BbbN^n$ denotar\'{a} el monoide libre conmutativo
con $n$ generadores $\epsilon_1, \dots, \epsilon_n$. Los elementos
de $\BbbN^n$, que llamaremos usualmente \emph{multi-\'{\i}ndices},
estar\'{a}n representados por vectores $\alpha = (\alpha_1, \dots,
\alpha_n)$ con componentes enteras no negativas, siendo la
operaci\'{o}n del monoide la adici\'{o}n usual de vectores. De esta forma,
los generadores $\epsilon_1, \dots, \epsilon_n$ son los vectores
de la base can\'{o}nica. Aunque haremos uso de monoides de diferentes
dimensiones, la notaci\'{o}n para sus generadores no variar\'{a}, ya que
el contexto no deja dudas en cada caso. El soporte $\sop{\alpha}$
de un multi-\'{\i}ndice $\alpha$ se define como el conjunto de los
\'{\i}ndices $i = 1, \dots, n$ tales que $\alpha_i \neq 0$.

\begin{definition}
Diremos que una relaci\'{o}n de orden total $\preceq$ en $\BbbN^n$ es
un \emph{orden admisible} si
\begin{enumerate}
\item Si $\alpha \preceq \beta$, entonces $\alpha + \gamma \preceq
\beta + \gamma$, para $\alpha, \beta, \gamma \in \BbbN^n$.
\item $0 \preceq \alpha$ para todo $\alpha \in \BbbN^n$.
\end{enumerate}
Usaremos la notaci\'{o}n $(\BbbN^n,\preceq)$ para representar esta
situaci\'{o}n.
\end{definition}

Una observaci\'{o}n fundamental es que, por el Lema de Dickson, todo
orden admisible hace de $\BbbN^n$ un conjunto bien ordenado.
Aunque el \'{u}nico orden admisible sobre $\BbbN$ es el usual, la
abundancia de \'{o}rdenes admisibles par $n > 1$ est\'{a} garantizada por
el hecho de que hay una cantidad no numerable de ellos.

A partir de este momento, $K$ designar\'{a} un anillo conmutativo y
$\mathbf{k}$ un cuerpo (conmutativo).

\begin{definition}
Una \emph{$(\BbbN^n,\preceq)$--filtraci\'{o}n} sobre una $K$--\'{a}lgebra
$R$ es una familia de $K$--subm\'{o}dulos $F(R) = \{ F_{\alpha}(R) ~|~
\alpha \in \BbbN^n \}$ de $R$ satisfaciendo los siguientes
axiomas:
\begin{enumerate}
\item
Si $\alpha \preceq \beta$, entonces $F_{\alpha}(R) \subseteq
F_{\beta}(R)$.
\item
Para cualesquiera $\gamma, \delta \in \BbbN^n$, se tiene
$F_{\gamma}(R)F_{\delta}(R) \subseteq F_{\gamma + \delta}(R)$.
\item
$\bigcup_{\alpha \in \BbbN^n} F_{\alpha}(R) = R$.
\item $1 \in F_0(R)$.
\end{enumerate}
Muchas veces relajaremos la notaci\'{o}n $F(R)$ y pondremos,
simplemente,  $F$.
\end{definition}

Cuando $n =1$, la anterior definici\'{o}n coincide, exactamente, con
la definici\'{o}n usual de filtraci\'{o}n positiva. Al igual que en el
caso de filtraciones, tenemos la noci\'{o}n de m\'{o}dulos multi-filtrado.

\begin{definition}
Supongamos dada una $(\BbbN^n,\preceq)$--filtraci\'{o}n $F(R)$ sobre
$R$ y sea $M$ un $R$--m\'{o}dulo por la izquierda. Una
$(\BbbN^n,\preceq)$--filtraci\'{o}n $F(M)$ sobre $M$ es una familia
$F(M) = \{ F_{\alpha}(M) ~|~ \alpha \in \BbbN^n \}$ de
$K$--subm\'{o}dulos de $M$ satisfaciendo las siguientes condiciones:
\begin{enumerate}
\item
Si $\alpha \preceq \beta$, entonces $F_{\alpha}(M) \subseteq
F_{\beta}(M)$.
\item
Para cualesquiera $\gamma, \delta \in \BbbN^n$, se tiene que
$F_{\gamma}(R)F_{\delta}(M) \subseteq F_{\gamma + \delta}(M)$.
\item $\bigcup_{\alpha \in \BbbN^n}F_{\alpha}(M) = M$.
\end{enumerate}
\end{definition}

La noci\'{o}n de \'{a}lgebra y m\'{o}dulo graduado asociado a los respectivos
objetos filtrados ser\'{a} fundamental para nosotros. La construcci\'{o}n
es bastante natural si escribimos para cada $\gamma \in \BbbN^n$,
\[
\blackhole{F}{\gamma}{M} = \bigcup_{\gamma' \prec
\gamma}F_{\gamma'}(M)
\]
para un m\'{o}dulo multi-filtrado $M$ (entendemos que
$\blackhole{F}{0}{M} = \{ 0 \}$). Consideremos el $K$--m\'{o}dulo
\[
G_{\gamma}(M) = \frac{F_{\gamma}(M)}{\blackhole{F}{\gamma}{M}}
\]
y definamos el $K$--m\'{o}dulo $\BbbN^n$--graduado
\[
G(M) = \oplus_{\gamma \in \BbbN^n}G_{\gamma}(M)
\]
Para $r + \blackhole{F}{\gamma}{R} \in G_{\gamma}(R)$ y $m +
\blackhole{F}{\delta}{M} \in G_{\delta}(M)$, definamos
\[
(r + \blackhole{F}{\gamma}{R})( m + \blackhole{F}{\delta}{M}) = rm
+ \blackhole{F}{\gamma + \delta}{M}
\]
Si $M = R$, entonces tenemos un producto en $G(R)$ que lo hace un
\'{a}lgebra $\BbbN^n$--graduada sobre $K$. Adem\'{a}s, $G(M)$ se convierte
as\'{\i} en un $G(R)$--m\'{o}dulo por la izquierda $\BbbN^n$--graduado.
Llamaremos a $G(R)$ (resp. a $G(M)$), \emph{\'{a}lgebra (resp. m\'{o}dulo)
graduado asociado}. Cuando sea necesario, subrayaremos la
dependencia de esta construcci\'{o}n con respecto de la
multi-filtraci\'{o}n usando la notaci\'{o}n $\graded{F}{R}$ o
$\graded{F}{M}$.

\'{I}ntimamente ligado a la noci\'{o}n de m\'{o}dulo graduado asociado se
encuentra el concepto de multi-grado de un elemento de $M$.
Concretamente, dado $m \in M$, llamamos \emph{multi-grado} de $0
\neq  m \in M$, notaci\'{o}n $\mdeg{F}{m}$, al m\'{\i}nimo (con respecto de
$\preceq$) de los multi-\'{\i}ndices $\alpha \in \BbbN^n$ tales que $m
\in F_{\alpha}(M)$. Es de rese\~{n}ar que $m +
\blackhole{F}{\mdeg{F}{m}}{M}$ es un elemento homog\'{e}neo de grado
$\mdeg{F}{m}$ de $G(M)$.

El primer resultado esperanzador sobre la utilidad del m\'{e}todo de
las multi-filtraciones es el siguiente (ver \cite[Theorem 1.5]
{Gomez:1999} para su (sencilla) demostraci\'{o}n).

\begin{theorem}(Teorema de la Base de Hilbert)
Sea $R$ un anillo multi-filtrado tal que $G(R)$ es noetheriano por
la izquierda. Entonces $R$ es noetheriano por la izquierda.
\end{theorem}

Otro resultado motivador para intentar dotar de multi-filtraciones
a las \'{a}lgebras de manera que el \'{a}lgebra multi-graduada asociada
sea sencilla es el siguiente (ver \cite[Theorem 2.8]{Gomez:1999}).

\begin{theorem}
Sea $M$ un m\'{o}dulo por la izquierda multi-filtrado sobre un \'{a}lgebra
multi-filtrada $R$. Supongamos que $G(R)$ es un \'{a}lgebra
finitamente generada y que $G(M)$ es un $G(R)$--m\'{o}dulo por la
izquierda finitamente generado. Entonces
\[
\GKdim({}_RM) \geqslant \GKdim({}_{G(R)}G(M))
\]
Si, adem\'{a}s, las multi-filtraciones son finito-dimensionales,
entonces
\[
\GKdim({}_RM) = \GKdim({}_{G(R)}G(M))
\]
\end{theorem}
Este \'{u}ltimo resultado es fundamental para obtener la exactitud de
la dimensi\'{o}n de Gelfand-Kirillov en una amplia clase de \'{a}lgebras
(ver \cite[Theorem 2.10]{Gomez:1999}).

\section{Multi-filtraciones y extensiones acotadas}

Comenzamos estudiando un m\'{e}todo de multi-filtraci\'{o}n para ciertas
extensiones de \'{a}lgebras.

Dados elementos $x_1, \dots, x_s$ de un anillo $R$ y un
\emph{multi-\'{\i}ndice} de componentes enteras no negativas $\gamma =
(\gamma_1,\dots,\gamma_s) \in \BbbN^s$, usaremos la notaci\'{o}n
$\mono{x}{\gamma} = x_1^{\gamma_1}\cdots x_s^{\gamma_s}$ y
llamaremos a estos elementos \emph{monomios est\'{a}ndar} en $x_1,
\dots, x_s$.

 Consideremos monoides ordenados $(\BbbN^m, \preceq_1)$,
$(\BbbN^s, \preceq_2)$ y $(\BbbN^n, \preceq)$, con $\preceq_1,
\preceq_2$ y $\preceq$ \'{o}rdenes admisibles, y morfismos de monoides
ordenados $\pcomp : \BbbN^m \rightarrow \BbbN^n$ y $\scomp :
\BbbN^s \rightarrow \BbbN^n$. Sea $A \subseteq B$ una extensi\'{o}n de
$K$--\'{a}lgebras, para $K$ un anillo conmutativo verificando las
siguientes condiciones:
\begin{enumerate}[({EA}1)]
\item\label{EA1} La $K$--\'{a}lgebra $A$ est\'{a} dotada de una
$(\BbbN^m, \preceq_1)$--filtraci\'{o}n $F = \{F_{\alpha}(A) ~|~ \alpha
\in \BbbN^m \}$.
\item\label{EA2} La $K$--\'{a}lgebra $B$ est\'{a} generada por $A$ junto con una
cantidad finita de elementos $x_1, \dots, x_s \in B$.
\item\label{EA3} Para cada $i = 1, \dots, s$ y cada $\alpha \in \BbbN^m$ se
tiene que
\[x_i F_{\alpha}(A) \subseteq F_{\alpha}(A)x_i +
\sum_{\peso{\gamma} \prec \pcomp(\alpha) + \scomp (\epsilon_i)}
F_{\gamma_1}(A)\mono{x}{\gamma_2}
\]
\item\label{EA4}
Para cada $1 \leqslant i < j \leqslant s$ existe $q_{ji} \in
F_{0}(A)$ tal que
\[ x_jx_i - q_{ji}x_ix_j \in \sum_{\peso{\gamma} \prec \scomp
(\epsilon_i + \epsilon_j)}F_{\gamma_1}(A)\mono{x}{\gamma_2}
\]
\end{enumerate}

\begin{definition}\label{acotada}
Diremos que la $B$ es una \emph{extensi\'{o}n
$(\pcomp,\scomp)$--acotada por la izquierda} del \'{a}lgebra
multi-filtrada $A$ cuando las condiciones (EA1), (EA2), (EA3) y
(EA4) son satisfechas.
\end{definition}

En la Definici\'{o}n \ref{acotada} se admite el caso $m = 0$,
entendiendo que $\BbbN^0$ es el semigrupo trivial y la filtraci\'{o}n
sobre $A$ es trivial (si se quiere, $A$ no se supone en tal caso
dotado de filtraci\'{o}n alguna). En tal caso, $\pcomp = 0$.
Destacaremos dos casos particulares de extensiones
$(\pcomp,\scomp)$--acotadas por la izquierda.

\begin{definition}
Supongamos $n = m+s$ en la Definici\'{o}n \ref{acotada}, y
pongamos $\pcomp : \BbbN^m \rightarrow \BbbN^{m+s}$ y $\scomp :
\BbbN^s \rightarrow \BbbN^{m+s}$ definidas por $\pcomp(\alpha) =
(\alpha,0)$, $\scomp(\beta) = (0,\beta)$. Si definimos $\preceq$
como
\[
(\alpha_1,\alpha_2) \preceq (\beta_1,\beta_2) \Leftrightarrow
\begin{cases}
\alpha_2 \prec_2 \beta_2 \\
or \\
\alpha_2 = \beta_2 \mbox{ and } \alpha_1 \preceq_1 \beta_1
\end{cases}
\]
entonces las condiciones (EA3) y (EA4) se escriben entonces como
\\ 1. Para cada $i = 1, \dots, n$ y cada $\alpha \in \BbbN^m$ se tiene que
\[
x_iF_{\alpha}(A) \subseteq F_{\alpha}(A)x_i + \sum_{\gamma \prec_2
\epsilon_i}A\mono{x}{\gamma}
\]
2. Para cada $1 \leqslant i < j \leqslant n$ existe $q_{ji} \in
A$ tal que
\[
x_jx_i - q_{ji}x_ix_j \in \sum_{\gamma \prec_2 \epsilon_i +
\epsilon_j}A\mono{x}{\gamma}
\]
La extensi\'{o}n $B$ se llamar\'{a} una \emph{extensi\'{o}n
$\preceq$--acotada por la izquierda} de $A$. En el caso de que
$q_{ji} \in U(A)$, el grupo de unidades de $A$, para todo $1
\leqslant i < j \leqslant s$, diremos que $B$ es una extensi\'{o}n
\emph{cu\'{a}ntica} $\preceq$--acotada de $A$. Observemos que si
$A \subseteq Cen(B)$, el centro de $B$, (por ejemplo, si $A = K$),
entonces la \'{u}nica condici\'{o}n relevante es la segunda.
\end{definition}

\begin{definition}
Si tomamos $m = 0$ y $n= 1$ en la Definici\'{o}n \ref{acotada},
entonces $\scomp$ proporciona el vector $\mathbf{w} = (w_1, \dots,
w_s) \in \BbbN^s$ dado por $w_i = \scomp(\epsilon_i)$ para $i = 1,
\dots s$. En tal caso, $\pcomp = \esc{\mathbf{w}}{-}$, donde
$\esc{-}{-}$ denota el producto escalar usual, y, para un monomio
est\'{a}ndar $\mono{x}{\gamma}$, $\scomp(\gamma) =
\esc{\mathbf{w}}{\gamma} = w_1\gamma_1 + \cdots + w_s\gamma_s$ es
su grado total $\mathbf{w}$--ponderado. Las condiciones (EA3) y
(EA4) se escriben entonces como
\\1. Para cada $i = 1, \dots, s$ se tiene que
\[
x_iA \subseteq Ax_i + \sum_{\esc{\mathbf{w}}{\gamma} <
w_i}A\mono{x}{\gamma}
\]
2. Para cada $1 \leqslant i < j \leqslant s$ existe $q_{ji} \in
A$ tal que
\[
x_jx_i - q_{ji}x_ix_j \in \sum_{\esc{\mathbf{w}}{\gamma} < w_i +
w_j }A\mono{x}{\gamma}
\]
Observemos que el papel del orden $\preceq_2$ se torna, en este
caso, un tanto irrelevante puesto que, dado cualquier $\mathbf{w}
\in \BbbN^s$, siempre podemos escoger $\preceq_2 =
\preceq_{\mathbf{w}}$ y $\scomp = \esc{\mathbf{w}}{-}$. En este
caso diremos que $B$ es una extensi\'{o}n \emph{$\mathbf{w}$--acotada
por la izquierda} de $A$.
\end{definition}

Nuestro primer objetivo es demostrar la siguiente proposici\'{o}n.

\begin{proposition}\label{esfiltracion}
Sea $A \subseteq B$ una extensi\'{o}n $(\pcomp,\scomp)$--acotada. Para
cada $\alpha \in \BbbN^n$, definamos el $K$--subm\'{o}dulo de $B$
\[
\filext{F}_{\alpha}(B) = \sum_{\peso{\alpha} \preceq
\alpha}F_{\alpha_1}(A)\mono{x}{\alpha_2}
\]
Entonces $\mathfrak{F} = \{ \filext{F}_{\alpha}(B) ~|~ \alpha \in
\BbbN^n \}$ es una $(\BbbN^n, \preceq)$--filtraci\'{o}n sobre $B$.
\end{proposition}
\begin{proof}
De la misma definici\'{o}n de $\mathfrak{F}$ se sigue que
\begin{equation}
\filext{F}_{\alpha}(B) \subseteq \filext{F}_{\beta}(B) \qquad
\textrm{si} \qquad \alpha \preceq \beta
\end{equation}
Adem\'{a}s, la inclusi\'{o}n $F_{0}(A) \subseteq \filext{F}_0(B)$
garantiza que $1 \in \filext{F}_0(B)$. Para completar una
demostraci\'{o}n de la Proposici\'{o}n \ref{esfiltracion}, hemos de
demostrar que
\begin{equation}\label{monomorfo}
\filext{F}_{\gamma}(B) \filext{F}_{\delta}(B) \subseteq
\filext{F}_{\gamma + \delta}(B)
\end{equation}
y que
\begin{equation}\label{exhaustiva}
\bigcup_{\alpha \in \BbbN^n} \filext{F}_{\alpha}(B) = B.
\end{equation}

 Para demostrar \eqref{monomorfo} y \eqref{exhaustiva} vamos a necesitar
 un lema t\'{e}cnico previo. Introducimos para ello el orden admisible $\preceq'$ sobre
$\BbbN^s \times \BbbN^n$ proporcionado por la siguiente definici\'{o}n
\begin{equation}
(\lambda, \nu) \preceq' (\lambda', \nu') \Leftrightarrow
\begin{cases}
\nu \prec \nu' & \\ \mathrm{\'{o}} & \\ \nu = \nu' \qquad \mathrm{y}
\qquad \lambda \preceq_2 \lambda'
\end{cases}
\end{equation}
para $(\lambda, \nu), (\lambda', \nu') \in \BbbN^s \times
\BbbN^n$.

La notaci\'{o}n $a^{\beta}$ indicar\'{a} un elemento de $A$ que pertenece
a $F_{\beta}(A)$, para $\beta \in
\BbbN^m$.

\begin{lemma}\label{conmutador}
 Dados $(\alpha_1, \alpha_2), (\beta_1,\beta_2)
\in
\BbbN^m \times
\BbbN^s$ escribamos $\alpha = \peso{\alpha}$ y $\beta =
\peso{\beta}$. Para cualesquiera elementos
\[
r = r^{\alpha_1}\mono{x}{\alpha_2} + \sum_{(\mu_2,\peso{\mu})
\prec' (\alpha_2, \alpha)}r^{\mu_1}\mono{x}{\mu_2} \]
 y
\[
s = s^{\beta_1}\mono{x}{\beta_2} + \sum_{(\nu_2,\peso{\nu}) \prec'
(\beta_2,\beta)}s^{\nu_1}\mono{x}{\nu_2}
\]
entonces
\[
rs = t^{\alpha_1 + \beta_1}\mono{x}{\alpha_2 + \beta_2} +
\sum_{(\omega_2,\peso{\omega}) \prec' (\alpha_2 + \beta_2, \alpha
+ \beta)} t^{\omega_1}\mono{x}{\omega_2}.
\]
\end{lemma}

\textit{Demostraci\'{o}n del lema.} Antes que nada, introduzcamos una
notaci\'{o}n adecuada: dado $\lambda = \peso{\lambda}$ para alg\'{u}n
$(\lambda_1,\lambda_2) \in \BbbN^m \times \BbbN^s$, el s\'{\i}mbolo
$\cola{\lambda_2}{\lambda}$ denotar\'{a} un elemento de la forma
$\sum_{(\rho_2, \peso{\rho}) \prec' (\lambda_2,
\lambda)}u^{\rho_1}\mono{x}{\rho_2}$ (entendemos que $\cola{0}{0}
= 0$). Con esta notaci\'{o}n, lo que pretendemos demostrar es que
\begin{equation}\label{cola0}
(r^{\alpha_1}\mono{x}{\alpha_2} +
\cola{\alpha_2}{\alpha})(s^{\beta_1}\mono{x}{\beta_2} +
\cola{\beta_2}{\beta}) = t^{\alpha_1 + \beta_1}\mono{x}{\alpha_2 +
\beta_2} + \cola{\alpha_2 + \beta_2}{\alpha + \beta}
\end{equation}

Sea $\Gamma$ el conjunto de los pares, $(\alpha_1, \alpha_2),
(\beta_1, \beta_2)
\in
\BbbN^m \times \BbbN^s$ tales que existen elementos $r =
r^{\alpha_1}\mono{x}{\alpha_2} + \cola{\alpha_2}{\alpha}$ y $s =
s^{\beta_1}\mono{x}{\beta_2} + \cola{\beta_2}{\beta}$ que no
satisfacen la igualdad \eqref{cola0}. Por supuesto, queremos
demostrar que $\Gamma$ es un conjunto vac\'{\i}o. Razonando por
reducci\'{o}n al absurdo, supongamos que $\Gamma$ fuese no vac\'{\i}o y
elijamos $(\alpha_1, \alpha_2), (\beta_1, \beta_2) \in \Gamma$
tales que $(\alpha_2, \alpha) + (\beta_2, \beta)$ es m\'{\i}nimo con
respecto de $\preceq'$, donde $\alpha = \peso{\alpha}$ y $\beta =
\peso{\beta}$, y sean $r,s \in B$ elementos que no satisfacen
\eqref{cola0}. La contradicci\'{o}n que deduciremos es que, despu\'{e}s de
todo, $r, s$ s\'{\i} satisfacen \eqref{cola0}. Si $(\alpha_2, \alpha) +
(\beta_2,\beta) = (0, 0)$, entonces
\[
rs = (r^{\alpha_1} + \cola{0}{0})(s^{\beta_1} + \cola{0}{0}) =
r^{\alpha_1}s^{\beta_1} = t^{\alpha_1 + \beta_1}
\]
por ser $\{ F_{\gamma}(A) ~|~ \gamma \in \BbbN^m \}$ una
multi-filtraci\'{o}n, lo que entra en contradicci\'{o}n con la elecci\'{o}n de
$(\alpha_1,\alpha_2), (\beta_1, \beta_2)$. Por tanto,
$(\alpha_2,\alpha) + (\beta_2,\beta) \neq (0,0)$ y, por
minimalidad, tenemos que para cualesquiera $(\gamma_1, \gamma_2),
(\delta_1, \delta_2)
\in
\BbbN^m \times \BbbN^s$ tales que $(\gamma_2, \gamma) + (\delta_2,
\delta) \prec' (\alpha_2, \alpha) + (\beta_2, \beta)$ con $\gamma
= \peso{\gamma}$ y $\delta = \peso{\delta}$ se tiene que
\begin{equation}\label{(a)}
(r^{\gamma_1}\mono{x}{\gamma_2} +
\cola{\gamma_2}{\gamma})(s^{\delta_1}\mono{x}{\delta_2} +
\cola{\delta_2}{\delta}) = t^{\gamma_1 +
\gamma_2}\mono{x}{\gamma_2 + \delta_2} + \cola{\gamma_2 +
\delta_2}{\gamma + \delta}
\end{equation}
Afirmamos que, entonces, siempre que $(\gamma_2,\gamma) +
(\delta_2,\delta) \preceq' (\alpha_2,\alpha) + (\beta_2,\beta)$,
se tiene
\begin{align}
\cola{\gamma_2}{\gamma}s^{\delta_1}\mono{x}{\delta_2} &=
\cola{\gamma_2 + \delta_2}{\gamma + \delta}
,\label{(b)}
\\ r^{\gamma_1}\mono{x}{\gamma_2}\cola{\delta_2}{\delta} &=
\cola{\gamma_2 + \delta_2}{\gamma + \delta} 
, \label{(c)}
\\ \cola{\gamma_2}{\gamma} \cola{\delta_2}{\delta} &=
\cola{\gamma_2 + \delta_2}{\gamma + \delta}
. \label{(d)}
\end{align}
En efecto, demostremos \eqref{(b)}. Por definici\'{o}n,
\begin{equation}\label{cola1}
\cola{\gamma_2}{\gamma} = \sum_{(\lambda_2,\peso{\lambda}) \prec'
(\gamma_2,\gamma)}a^{\lambda_1}\mono{x}{\lambda_2}.
\end{equation}
Sea $(\mu_2, \mu) \in \BbbN^s \times \BbbN^n$ el m\'{a}ximo, con
respecto de $\preceq'$ del conjunto finito cuyos elementos son los
de la forma $(\lambda_2, \peso{\lambda})$ con
$a^{\lambda_1}\mono{x}{\lambda_2} \neq 0$. Entonces la ecuaci\'{o}n
\eqref{cola1} se puede reescribir como
\begin{equation}\label{cola2}
\cola{\gamma_2}{\gamma} = \sum_{(\lambda_2, \peso{\lambda}) =
(\mu_2, \mu)} a^{\lambda_1}\mono{x}{\lambda_2} + \sum_{(\lambda_2,
\peso{\lambda}) \prec' (\mu_2, \mu)}
a^{\lambda_1}\mono{x}{\lambda_2}
\end{equation}
El primer sumando del segundo miembro en \eqref{cola2} es de la
forma
\[
(\sum_{\pcomp(\lambda_1) + \scomp(\mu_2) = \mu}
a^{\lambda_1})\mono{x}{\mu_2},
\]
luego, tomando $\mu_1$ el m\'{a}ximo con respecto de $\preceq_1$ de
los $\lambda_1$ que aparecen en dicha suma, tenemos que
$b^{\mu_1} =\sum_{\pcomp(\lambda_1) + \scomp(\mu_2) = \mu}
a^{\lambda_1} \in F_{\mu_1}(A)$ y, por tanto,
\begin{equation}\label{cola3}
\cola{\gamma_2}{\gamma} = b^{\mu_1}\mono{x}{\mu_2} +
\cola{\mu_2}{\mu}
\end{equation}
con $(\mu_2, \mu) \prec' (\gamma_2,\gamma)$. De aqu\'{\i},
\begin{equation*}
(\mu_2,\mu) + (\delta_2,\delta) \prec' (\gamma_2,\gamma) +
(\delta_2, \delta) \preceq' (\alpha_2,\alpha) + (\beta_2, \beta)
\end{equation*}
Por tanto, en vista de \eqref{(a)} y \eqref{cola3}, tenemos
\begin{equation*}
\cola{\gamma_2}{\gamma}s^{\delta_1}\mono{x}{\delta_2} =
t^{\mu_1}\mono{x}{\mu_2} + \cola{\mu_2 + \delta_2}{\mu + \delta} =
\cola{\gamma_2 + \delta_2}{\gamma + \delta}
\end{equation*}
Las demostraciones de \eqref{(c)} y \eqref{(d)} se pueden
construir an\'{a}logamente.

De \eqref{(b)}, \eqref{(c)} y \eqref{(d)} deducimos que
\begin{alignat*}{2}
rs &= (r^{\alpha_1}\mono{x}{\alpha_2} + \cola{\alpha_2}{\alpha}
)(s^{\beta_1}\mono{x}{\beta_2} + \cola{\beta_2}{\beta})
\\ &=
r^{\alpha_1}\mono{x}{\alpha_2}s^{\beta_1}\mono{x}{\beta_2} +
r^{\alpha_1}\cola{\beta_2}{\beta} +
\cola{\alpha_2}{\alpha}s^{\beta_1}\mono{x}{\beta_2} +
\cola{\alpha_2}{\alpha}\cola{\beta_2}{\beta}
\\ &= r^{\alpha_1}\mono{x}{\alpha_2}s^{\beta_1}\mono{x}{\beta_2} +
\cola{\alpha_2 + \beta_2}{\alpha + \beta} + \cola{\alpha_2 +
\beta_2}{\alpha + \beta} + \cola{\alpha_2 + \beta_2}{\alpha +
\beta}
\\ &= r^{\alpha_1}\mono{x}{\alpha_2}s^{\beta_1}\mono{x}{\beta_2} +
\cola{\alpha_2 + \beta_2}{\alpha + \beta}
\end{alignat*}
Esto muestra que para obtener que $r, s$ satisfacen \eqref{cola0},
podemos suponer sin p\'{e}rdida de generalidad que $r =
r^{\alpha_1}\mono{x}{\alpha_2}$ y $s =
s^{\beta_1}\mono{x}{\beta_2}$. Desde luego, si $\alpha_2 = 0$, no
tenemos nada que demostrar, as\'{\i} que supongamos que $\alpha_2 \neq
0$. En este caso, $\mono{x}{\alpha_2} = x_i\mono{x}{\alpha_2 -
\epsilon_i}$. Entonces

\begin{alignat*}{2}
\mono{x}{\alpha_2}s^{\beta_1} &= x_i\mono{x}{\alpha_2 -
\epsilon_i}s^{\beta_1}
\\ &= x_i(t^{\beta_1}\mono{x}{\alpha_2 - \epsilon_i} +
\cola{\alpha_2 - \epsilon_i}{\pcomp(\beta_1) + \scomp(\alpha_2 -
\epsilon_i})) & \qquad (\textrm{por \eqref{(a)}})
\\ &= x_it^{\beta_1}\mono{x}{\alpha_2 - \epsilon_i} +
x_i\cola{\alpha_2 - \epsilon_i}{\pcomp(\beta_1) + \scomp(\alpha_2
- \epsilon_i)}
\\ &= x_it^{\beta_1}\mono{x}{\alpha_2 - \epsilon_i} +
\cola{\alpha_2}{\pcomp(\beta_1) + \scomp(\alpha_2)} & \qquad
(\textrm{por \eqref{(c)}})
\\ &= (q^{\beta_1}x_i + \cola{\epsilon_i}{\pcomp(\beta_1) +
\scomp(\epsilon_i)})\mono{x}{\alpha_2 - \epsilon_i} +
\cola{\alpha_2}{\pcomp(\beta_1) + \scomp(\alpha_2)} & \qquad
(\textrm{por (EA3)})
\\ &= q^{\beta_1}x_i\mono{x}{\alpha_2 - \epsilon_i} +
\cola{\alpha_2}{\pcomp(\beta_1) + \scomp(\alpha_2)} & \qquad
(\textrm{por \eqref{(b)}})
\\ &= q^{\beta_1}\mono{x}{\alpha_2} +
\cola{\alpha_2}{\pcomp(\beta_1) + \scomp(\alpha_2)}
\end{alignat*}

De esta forma,
\begin{alignat*}{2}
r^{\alpha_1}\mono{x}{\alpha_2}s^{\beta_1}\mono{x}{\beta_2} &=
r^{\alpha_1}(q^{\beta_1}\mono{x}{\alpha_2} +
\cola{\alpha_2}{\pcomp(\beta_1) +
\scomp(\alpha_2)})\mono{x}{\beta_2}
\\ &= r^{\alpha_1}q^{\beta_1}\mono{x}{\alpha_2}\mono{x}{\beta_2} +
r^{\alpha_1}(\cola{\alpha_2}{\pcomp(\beta_1) +
\scomp(\alpha_2)}\mono{x}{\beta_2})
\\ &= s^{\alpha_1 + \beta_1}\mono{x}{\alpha_2}\mono{x}{\beta_2} +
r^{\alpha_1}\cola{\alpha_2 + \beta_2}{\pcomp(\beta_1) +
\scomp(\alpha_2 + \beta_2)} & \qquad (\textrm{por \eqref{(b)}})
\\ &= s^{\alpha_1 + \beta_1}\mono{x}{\alpha_2}\mono{x}{\beta_2} +
\cola{\alpha_2 + \beta_2}{\pcomp(\alpha_1 + \beta_1) +
\scomp(\alpha_2 + \beta_2)} & \qquad (\textrm{por \eqref{(c)}})
\end{alignat*}
Escribamos la igualdad reci\'{e}n obtenida para futura referencia
\begin{equation}\label{cola4}
r^{\alpha_1}\mono{x}{\alpha_2}s^{\beta_1}\mono{x}{\beta_2}  =
s^{\alpha_1 + \beta_1}\mono{x}{\alpha_2}\mono{x}{\beta_2} +
\cola{\alpha_2 + \beta_2}{\pcomp(\alpha_1 + \beta_1) +
\scomp(\alpha_2 + \beta_2)}
\end{equation}
El tramo final de esta demostraci\'{o}n consiste en obtener una
<<representaci\'{o}n est\'{a}ndar>> de
$\mono{x}{\alpha_2}\mono{x}{\beta_2}$, para lo que procedemos como
sigue: Podemos suponer, sin p\'{e}rdida de generalidad, que tanto
$\alpha_2$ como $\beta_2$ son multi-\'{\i}ndices no nulos, entonces
$\sop{\alpha_2}$ y $\sop{\beta_2}$ son subconjuntos no vac\'{\i}os de
$\{1, \dots, n \}$, conjunto \'{e}ste que consideramos ordenado con el
orden natural. Sea $h$ (resp. $i$) el m\'{\i}nimo de $\sop{\alpha_2}$
(resp. $\sop{\beta_2}$), y sea $j$ (resp. $k$) el m\'{a}ximo de
$\sop{\alpha_2}$ (resp. $\sop{\beta_2}$). Vamos a distinguir tres
casos.

\textbf{Caso $h \leq i$:} Entonces
\begin{alignat*}{2}
\mono{x}{\alpha_2}\mono{x}{\beta_2} &=  x_h\mono{x}{\alpha_2 -
\epsilon_h}\mono{x}{\beta}
\\ &= x_h(t^{0}\mono{x}{\alpha_2 - \epsilon_h + \beta_2} +
\cola{\alpha_2 - \epsilon_h + \beta_2}{\scomp(\alpha_2 -
\epsilon_h + \beta_2)}) & \qquad (\textrm{por \eqref{(a)}})
\\ &= x_ht^{0}\mono{x}{\alpha_2 - \epsilon_h + \beta_2} +
\cola{\alpha_2 + \beta_2}{\scomp(\alpha_2 + \beta_2)} & \qquad
(\textrm{por \eqref{(c)}})
\\ &= (q^{0}x_h +
\cola{\epsilon_h}{\scomp(\epsilon_h)})\mono{x}{\alpha_2 -
\epsilon_h + \beta_2} + \cola{\alpha_2 + \beta_2}{\scomp(\alpha_2
+ \beta_2)} & \qquad (\textrm{por (EA3)})
\\ &= q^{0}x_h\mono{x}{\alpha_2 + \beta_2} + \cola{\alpha_2 +
\beta_2}{\scomp(\alpha_2 + \beta_2)} & \qquad (\textrm{por
\eqref{(b)}})
\\ &= q^0\mono{x}{\alpha_2 + \beta_2} + \cola{\alpha_2 +
\beta_2}{\scomp(\alpha_2 + \beta_2)}
\end{alignat*}

\textbf{Caso $j \leqslant k$:} Aqu\'{\i},
\begin{alignat*}{2}
\mono{x}{\alpha_2}\mono{x}{\beta_2} &=
\mono{x}{\alpha_2}\mono{x}{\beta_2 - \epsilon_k}x_k
\\ &= (t^0\mono{x}{\alpha_2 + \beta_2 - \epsilon_k} +
\cola{\alpha_2 + \beta_2 - \epsilon_k}{\scomp(\alpha_2 + \beta_2
- \epsilon_k)})x_k & \qquad (\textrm{por \eqref{(a)}})
\\ &= t^0\mono{x}{\alpha_2 + \beta_2 - \epsilon_k}x_k +
\cola{\alpha_2 + \beta_2}{\scomp(\alpha_2 + \beta_2)} & \qquad
(\textrm{por \eqref{(b)}})
\\ &= t^o\mono{x}{\alpha_2 + \beta_2} + \cola{\alpha_2 +
\beta_2}{\scomp(\alpha_2 + \beta_2)}
\end{alignat*}

\textbf{Caso $h > i$ y $j > k$:} Observemos que, necesariamente,
$i < j$. Procedemos como sigue:
\begin{alignat*}{2}
\mono{x}{\alpha_2}\mono{x}{\beta_2} &= \mono{x}{\alpha_2 -
\epsilon_j}x_jx_i\mono{x}{\beta_2 - \epsilon_i}
\\ &= \mono{x}{\alpha_2 - \epsilon_j}(q_{ji}x_ix_j +
\cola{\epsilon_i + \epsilon_j}{\scomp(\epsilon_i +
\epsilon_j)})\mono{x}{\beta_2 - \epsilon_i} & \qquad (\textrm{por
(EA4)})
\\ &= \mono{x}{\alpha_2 - \epsilon_j}q_{ji}x_ix_j\mono{x}{\beta_2
- \epsilon_i} + \mono{x}{\alpha_2 - \epsilon_j}\cola{\epsilon_i +
\epsilon_j}{\scomp(\epsilon_i + \epsilon_j)}\mono{x}{\beta_2 -
\epsilon_i}
\\ &= (q^0\mono{x}{\alpha_2 - \epsilon_j} + \cola{\alpha_2 -
\epsilon_j}{\scomp(\alpha_2 - \epsilon_j)})x_ix_j\mono{x}{\beta_2
- \epsilon_i}
\\ &\hspace{10em}+ \mono{x}{\alpha_2 - \epsilon_j}\cola{\epsilon_i
+ \epsilon_j}{\scomp(\epsilon_i + \epsilon_j)}\mono{x}{\beta_2 -
\epsilon_i} & \qquad (\textrm{por (EA3)})
\\ &= q^0\mono{x}{\alpha_2-\epsilon_j}x_ix_j\mono{x}{\beta_2 -
\epsilon_i} + \cola{\alpha_2 - \epsilon_j}{\scomp(\alpha_2 -
\epsilon_j)}x_ix_j\mono{x}{\beta_2 - \epsilon_i}
\\ & \hspace{10em} + \cola{\alpha_i + \epsilon_i}{\scomp(\alpha_i +
\epsilon_i)}\mono{x}{\beta_2 - \epsilon_i} & \qquad (\textrm{por
\eqref{(c)}})
\\ &= q^0\mono{x}{\alpha_2-\epsilon_j}x_ix_j\mono{x}{\beta_2 -
\epsilon_i} + \cola{\alpha_2 - \epsilon_j}{\scomp(\alpha_2 -
\epsilon_j)}x_ix_j\mono{x}{\beta_2 - \epsilon_i}
\\ & \hspace{10em} + \cola{\alpha_2 + \beta_2}{\scomp(\alpha_2 +
\beta_2)} & \qquad (\textrm{por \eqref{(b)}})
\\ &=q^0\mono{x}{\alpha_2-\epsilon_j}x_ix_j\mono{x}{\beta_2 -
\epsilon_i} + \cola{\alpha_2 - \epsilon_j +
\epsilon_i}{\scomp(\alpha_2 - \epsilon_j +
\epsilon_i)}x_j\mono{x}{\beta_2 - \epsilon_i}
\\ & \hspace{10em} + \cola{\alpha_2 + \beta_2}{\scomp(\alpha_2 +
\beta_2)} & \qquad (\textrm{por \eqref{(b)}})
\\ &=q^0\mono{x}{\alpha_2-\epsilon_j}x_ix_j\mono{x}{\beta_2 -
\epsilon_i} + \cola{\alpha_2  + \epsilon_i}{\scomp(\alpha_2  +
\epsilon_i)}\mono{x}{\beta_2 - \epsilon_i}
\\ & \hspace{10em} + \cola{\alpha_2 + \beta_2}{\scomp(\alpha_2 +
\beta_2)} & \qquad (\textrm{por \eqref{(b)}})
\\ &=q^0\mono{x}{\alpha_2-\epsilon_j}x_ix_j\mono{x}{\beta_2 -
\epsilon_i}  + \cola{\alpha_2 + \beta_2}{\scomp(\alpha_2 +
\beta_2)} & \qquad (\textrm{por \eqref{(b)}})
\\ &= q^0(t^0\mono{x}{\alpha_2 - \epsilon_j + \epsilon_i} +
\cola{\alpha_2 - \epsilon_j + \epsilon_i}{\scomp(\alpha_2 -
\epsilon_j + \epsilon_i)})x_j\mono{x}{\beta_2 - \epsilon_i}
\\ & \hspace{10em} + \cola{\alpha_2 + \beta_2}{\scomp(\alpha_2 +
\beta_2)} & \qquad (\textrm{por \eqref{(a)}})
\\ &= u^0\mono{x}{\alpha_2 - \epsilon_j +
\epsilon_i}x_j\mono{x}{\beta_0 - \epsilon_i} + q^0\cola{\alpha_2
+ \epsilon_i}{\scomp(\alpha_2 +
\epsilon_i)}\mono{x}{\beta_2-\epsilon_i}
\\ & \hspace{10em} + \cola{\alpha_2 + \beta_2}{\scomp(\alpha_2 +
\beta_2)} & \qquad (\textrm{por \eqref{(b)}})
\\ &= u^0\mono{x}{\alpha_2 - \epsilon_j +
\epsilon_i}x_j\mono{x}{\beta_0 - \epsilon_i} + \cola{\alpha_2 +
\beta_2}{\scomp(\alpha_2 + \beta_2)} & \qquad (\textrm{por
\eqref{(c)} y \eqref{(b)}})
\\ &= u^0x_i\mono{x}{\alpha_2 - \epsilon_j}x_j\mono{x}{\beta_2- \epsilon_i} +
\cola{\alpha_2 + \beta_2}{\scomp(\alpha_2 + \beta_2)} & \qquad
(\textrm{ya que } i < j)
\\ &= u^0x_i\mono{x}{\alpha_2}\mono{x}{\beta_2 - \epsilon_i} +
\cola{\alpha_2 + \beta_2}{\scomp(\alpha_2 + \beta_2)}
\\ &= u^0x_i(s^0\mono{x}{\alpha_2 + \beta_2 - \epsilon_i} +
\cola{\alpha_2 + \beta_2 - \epsilon_i}{\scomp(\alpha_2 + \beta_2
- \epsilon_i)})
\\ & \hspace{10em} + \cola{\alpha_2 + \beta_2}{\scomp(\alpha_2 +
\beta_2)} & \qquad (\textrm{por \eqref{(a)}})
\\ &= u^0x_is^0\mono{x}{\alpha_2 + \beta_2 - \epsilon_i}
+ \cola{\alpha_2 + \beta_2}{\scomp(\alpha_2 + \beta_2)} &
\qquad(\textrm{por \eqref{(c)}})
\\ &= u^0(s^0x_i + \cola{\epsilon_i}{\scomp(\epsilon_i)}) +
\cola{\alpha_2 + \beta_2}{\scomp(\alpha_2 + \beta_2)} &
(\textrm{por (EA3)})
\\ &= u^0s^0x_i\mono{x}{\alpha_2 + \beta_2 - \epsilon_i} +
\cola{\alpha_2 + \beta_2}{\scomp(\alpha_2 + \beta_2)} & \qquad
(\textrm{por \eqref{(b)}})
\\ &= v^0\mono{x}{\alpha_2 + \beta_2} + \cola{\alpha_2 +
\beta_2}{\scomp(\alpha_2 + \beta_2)}
\end{alignat*}
Sustituyendo la igualdad obtenida en \eqref{cola4} tenemos
\begin{alignat*}{2}
r^{\alpha_1}\mono{x}{\alpha_2}s^{\beta_1}\mono{x}{\beta_2} &=
s^{\alpha_1 + \beta_1}(v^0\mono{x}{\alpha_2 + \beta_2} +
\cola{\alpha_2 + \beta_2}{\scomp(\alpha_2 + \beta_2)})
\\ & \hspace{5em} +
\cola{\alpha_2 + \beta_2}{\pcomp(\alpha_1 + \beta_1) +
\scomp(\alpha_2 + \beta_2)}
\\ &= s^{\alpha_1 + \beta_1}v^0\mono{x}{\alpha_2 + \beta_2} +
\cola{\alpha_2 + \beta_2}{\pcomp(\alpha_1 + \beta_1) +
\scomp(\alpha_2 + \beta_2)} & \qquad (\textrm{por \eqref{(c)}})
\\ &= t^{\alpha_1 + \beta_1}\mono{x}{\alpha_2 + \beta_2} +
\cola{\alpha_2 + \beta_2}{\pcomp(\alpha_1 + \beta_1)+
\scomp(\alpha_2 + \beta_2)}
\end{alignat*}
Con esto hemos deducido que los elementos $r, s$ satisfacen
\eqref{cola0}, lo que concluye la demostraci\'{o}n del lema.

\bigskip

Ahora estamos en condiciones de terminar la demostraci\'{o}n de la
Proposici\'{o}n \ref{esfiltracion}. Comencemos por la inclusi\'{o}n
\eqref{monomorfo}. Dados $r \in \filext{F}_{\gamma}(B)$ y $s \in
\filext{F}_{\delta}(B)$, tenemos que
\begin{equation*}
r = \sum_{\peso{\gamma} \preceq
\gamma}r^{\gamma_1}\mono{x}{\gamma_2}, \qquad s =
\sum_{\peso{\delta} \preceq \delta} s^{\delta_1}\mono{x}{\delta_2}
\end{equation*}
Sea
\begin{equation*}
\alpha = \max_{\preceq} \{ \peso{\gamma} ~|~
r^{\gamma_1}\mono{x}{\gamma_1} \neq 0 \}
\end{equation*}
 y
\begin{equation*}
 \beta =
\max_{\preceq} \{ \peso{\delta} ~|~
s^{\delta_1}\mono{x}{\delta_2} \neq 0 \}
\end{equation*}
Claramente, $\alpha \preceq \gamma$ y $\beta \preceq \delta$, de
donde $\alpha + \beta \preceq \gamma + \delta$. Dado que
$\filext{F}_{\alpha + \beta}(B) \subseteq \filext{F}_{\gamma +
\delta}(B)$, bastar\'{a} con demostrar que $rs \in \filext{F}_{\alpha
+ \beta}(B)$ para obtener \eqref{monomorfo}. Obviamente, el
conjunto de los $\gamma_2 \in \BbbN^s$ tales que $\peso{\gamma} =
\alpha$ y $r^{\gamma_1}\mono{x}{\gamma_2} \neq 0$ para alg\'{u}n
$\gamma_1 \in \BbbN^s$ es finito. Sea $\alpha_2 \in \BbbN^s$ su
m\'{a}ximo con respecto de $\preceq_2$. Entonces
\begin{equation}\label{jefe}
r = \sum_{\pcomp(\gamma_1) + \scomp(\alpha_2) =
\alpha}r^{\gamma_1}\mono{x}{\alpha_2} + \sum_{\peso{\gamma}
\preceq \alpha \atop \gamma_2 \prec_2
\alpha_2}r^{\gamma_1}\mono{x}{\gamma_2}
\end{equation}
Tomando $\alpha_1 \in \BbbN^s$ el m\'{a}ximo entre los $\gamma_1$
tales que $\pcomp(\gamma_1) + \scomp(\alpha_2) = \alpha$ y
$r^{\gamma_1}\mono{x}{\alpha_2} \neq 0$, la igualdad
$\eqref{jefe}$ se escribe
\begin{equation}\label{jefea}
r = r^{\alpha_1}\mono{x}{\alpha_2} + \sum_{(\gamma_2,
\peso{\gamma}) \prec' (\alpha_2,
\alpha)}r^{\gamma_1}\mono{x}{\gamma_2} \qquad (\alpha =
\peso{\alpha})
\end{equation}
An\'{a}logamente, tenemos una expresi\'{o}n
\begin{equation}\label{jefeb}
s = s^{\beta_1}\mono{x}{\beta_2} + \sum_{(\delta_2,\peso{\delta})
\prec' (\beta_2, \beta)} s^{\delta_1}\mono{x}{\delta_2} \qquad
(\beta = \peso{\beta})
\end{equation}
Las ecuaciones \eqref{jefea} y \eqref{jefeb} muestran, en vista
del Lema \ref{conmutador}, que $rs \in \filext{F}_{\alpha +
\beta}(B)$, de donde \eqref{monomorfo}.

Para concluir que $\filext{F}$ es una
$(\BbbN^n,\preceq)$--filtraci\'{o}n sobre $B$ resta por comprobar
\eqref{exhaustiva}. Como $B$ est\'{a} generado como anillo por $A$ y
$x_1, \dots, x_s$, todo elemento de $B$ es una suma de
<<monomios>> de la forma $a_1x_{i_1}a_2x_{i_2}\cdots
x_{i_t}a_{t+1}$, con $a_{j} \in A$. Sean $\mu_1, \dots, \mu_t \in
\BbbN^s$ tales que $a_j \in F_{\mu_j}(A)$. Deducimos de
\eqref{monomorfo}, que ya hemos demostrado, que
\begin{equation*}
a_1x_{i_1}a_2x_{i_2} \cdots x_{i_t}a_t \in
\filext{F}_{\pcomp(\mu_1 + \cdots + \mu_t) + \scomp(\epsilon_{i_1}
+ \cdots \epsilon_{i_t})}(B),
\end{equation*}
lo que proporciona \eqref{exhaustiva} y, a la postre, la
demostraci\'{o}n de la Proposici\'{o}n \ref{esfiltracion}.
\end{proof}

Una vez dotada la $K$--\'{a}lgebra $B$ de la multi-filtraci\'{o}n
$\filext{F}$, vamos a estudiar la relaci\'{o}n de su \'{a}lgebra graduada
asociada $\graded{\filext{F}}{B}$ con $\graded{F}{A}$. Para ello,
supondremos que el homomorfismo de monoides ordenados $\pcomp :
\BbbN^m \rightarrow \BbbN^n$ es una aplicaci\'{o}n inyectiva. Esto
permite demostrar que $\blackhole{F}{\alpha}{A} \subseteq
\blackhole{\filext{F}}{\pcomp(\alpha)}{B}$ para todo $\alpha \in
\BbbN^m$. De aqu\'{\i}, la aplicaci\'{o}n $f : \graded{F}{A} \rightarrow
\graded{\filext{F}}{B}$ definida sobre componentes homog\'{e}neas por
\[
f(a + \blackhole{F}{\alpha}{A}) = a +
\blackhole{\filext{F}}{\pcomp(\alpha)}{B} \qquad (a +
\blackhole{F}{\alpha}{A} \in G^{F}_{\alpha}(R))
\]
 es un homomorfismo
$\pcomp$--graduado de $K$--\'{a}lgebras. Por tanto, la imagen de $f$
es una sub\'{a}lgebra $\BbbN^n$--graduada de $\graded{\filext{F}}{B}$,
que denotaremos por $G(A)$, donde, para cada $\beta \in \BbbN^n$,
la componente $\beta$--homog\'{e}nea viene dada por
\[
G_{\beta}(A) = \begin{cases} \frac{F_{\alpha}(A) +
\blackhole{\filext{F}}{\beta}{B}}{\blackhole{\filext{F}}{\beta}{B}}
& \textrm{si } \beta = \pcomp(\alpha) \\ 0 & \textrm{si } \beta
\notin \pcomp(\BbbN^m)
\end{cases}
\]

Con esta notaci\'{o}n, tenemos

\begin{theorem}\label{relacionesgr}
La $K$--\'{a}lgebra $\graded{\filext{F}}{B}$ est\'{a} generada por $G(A)$
y los elementos homog\'{e}neos $y_1, \dots, y_s$, donde $y_i = x_i +
\blackhole{\filext{F}}{\scomp(\epsilon_i)}{B}$ para $i = 1,
\dots, s$. Adem\'{a}s, se verifican los siguientes enunciados:
\begin{enumerate}
\item $G_{\alpha}^{\filext{F}}(B) = \sum_{\peso{\alpha} =
\alpha}G_{\pcomp(\alpha_1)}(A)\mono{y}{\alpha_2} \qquad (\alpha
\in \BbbN^n)$,
\item $y_iG_{\pcomp(\gamma)}(A) \subseteq G_{\pcomp(\gamma)}(A)y_i \qquad (\gamma
\in \BbbN^m, i = 1, \dots, s)$,
\item $y_jy_i = q_{ji}y_iy_j \qquad (1 \leqslant i < j \leqslant
s)$.
\end{enumerate}
\end{theorem}
\begin{proof}
Antes que nada, observemos que, para cada $\beta \in \BbbN^s$,
tenemos que
\begin{equation}\label{monogr}
\mono{y}{\beta} = \mono{x}{\beta} +
\blackhole{\filext{F}}{\scomp(\beta)}{B}
\end{equation}
1. Dado $(\alpha_1, \alpha_2) \in \BbbN^m \times \BbbN^s$,
escribamos $\alpha = \peso{\alpha}$ y tomemos $a +
\blackhole{\filext{F}}{\pcomp(\alpha_1)}{B} \in
G_{\pcomp(\alpha_1)}(A)$. Entonces, por \eqref{monogr},
\[
(a +
\blackhole{\filext{F}}{\pcomp(\alpha_1)}{B})\mono{y}{\alpha_2} =
a\mono{x}{\alpha_2} + \blackhole{\filext{F}}{\peso{\alpha}}{B},
\]
de donde $\sum_{\peso{\alpha}}
G_{\pcomp(\alpha_1)}(A)\mono{y}{\alpha_2} \subseteq
G^{\filext{F}}_{\alpha}(B)$. Para obtener la inclusi\'{o}n rec\'{\i}proca,
tomemos $h = g + \blackhole{\filext{F}}{\alpha}{B} \in
G^{\filext{F}}_{\alpha}(B)$, donde $g = \sum_{\peso{\alpha} =
\alpha}r^{\alpha_1}\mono{x}{\alpha_2}$. Si $h \neq 0$, entonces
podemos suponer que $r^{\alpha_1}\mono{x}{\alpha_2} \notin
\blackhole{\filext{F}}{\alpha}{B}$ para todos los
$(\alpha_1,\alpha_2) \in \BbbN^m \times \BbbN^s$. Esto permite
escribir
\begin{multline*}
g  = \sum_{\peso{\alpha} = \alpha}(r^{\alpha_1}\mono{x}{\alpha_2}
+ \blackhole{\filext{F}}{\alpha}{B}) \\ =  \sum_{\peso{\alpha} =
\alpha}(r^{\alpha_1} +
\blackhole{\filext{F}}{\pcomp(\alpha_1)}{B})(\mono{x}{\alpha_2} +
\blackhole{\filext{F}}{\scomp(\alpha_2)}{B}) \\ =
\sum_{\peso{\alpha} = \alpha}(r^{\alpha_1} +
\blackhole{\filext{F}}{\pcomp(\alpha_1)}{B})\mono{y}{\alpha_2},
\end{multline*}
de donde $h \in \sum_{\peso{\alpha} =
\alpha}G_{\pcomp(\alpha_1)}(A)\mono{y}{\alpha_2}$. \\

2. Usando (EA3), obtenemos
\begin{multline*}
y_i(r^{\gamma} + \blackhole{\filext{F}}{\pcomp(\gamma)}{B}) = (x_i
+ \blackhole{\filext{F}}{\scomp(\epsilon_i)}{B})(r^{\gamma} +
\blackhole{\filext{F}}{\pcomp(\gamma)}{B}) \\ = x_ir^{\gamma} +
\blackhole{\filext{F}}{\scomp(\epsilon_i) + \pcomp(\gamma)}{B} =
s^{\gamma}x_i + \blackhole{\filext{F}}{\pcomp(\gamma) +
\scomp(\epsilon_i)}{B} \\ = (s^{\gamma} +
\blackhole{\filext{F}}{\pcomp(\gamma)}{B})(x_i +
\blackhole{\filext{F}}{\scomp(\epsilon_i)}{B}) = (s^{\gamma} +
\blackhole{\filext{F}}{\pcomp(\gamma)}{B})x_i
\end{multline*}

3. Usando (EA4), tenemos
\begin{multline*}
y_jy_i = (x_j +
\blackhole{\filext{F}}{\scomp(\epsilon_j)}{B})(x_i +
\blackhole{\filext{F}}{\scomp(\epsilon_i)}{B}) \\= x_jx_i +
\blackhole{\filext{F}}{\scomp(\epsilon_i) \scomp(\epsilon_j)}{B}
= q_{ji}x_ix_j + \blackhole{\filext{F}}{\scomp(\epsilon_i) +
\scomp(\epsilon_j)}{B} \\= q_{ji}(x_i +
\blackhole{\filext{F}}{\scomp(\epsilon_i)}{B})(x_j +
\blackhole{\filext{F}}{\scomp(\epsilon_j)}{B}) = q_{ji}y_iy_j
\end{multline*}
\end{proof}

\begin{lemma}\label{noinyl}
El homomorfismo $f : \graded{F}{A} \rightarrow
\graded{\filext{F}}{B}$ es inyectivo si y s\'{o}lo si para todo $a
\in A$ no existe ninguna expresi\'{o}n del tipo
\begin{equation}\label{noiny}
a = \sum_{\peso{\beta} \prec \pcomp(\mdeg{F}{a})}
a^{\beta_1}\mono{x}{\beta_2}
\end{equation}
\end{lemma}
\begin{proof}
Supongamos que existe una expresi\'{o}n como \eqref{noiny} para alg\'{u}n
$a \in A$ y escribamos $\alpha = \mdeg{F}{a}$. Entonces $a  \in
\blackhole{\filext{F}}{\pcomp(\alpha)}{B}$, de donde $f(a +
\blackhole{F}{\alpha}{A}) = 0$. Pero $a +
\blackhole{F}{\alpha}{A} \neq 0$. Rec\'{\i}procamente, supongamos que
$f$ no es inyectivo y sea $a + \blackhole{F}{\mdeg{F}{\alpha}}{A}$
un elemento homog\'{e}neo no nulo en el n\'{u}cleo de $f$. Eso significa
que $a \in \blackhole{\filext{F}}{\pcomp(\mdeg{F}{\alpha})}{B}$, o
sea, que existe $\beta \prec \pcomp(\mdeg{F}{\alpha})$ tal que $a
\in \filext{F}_{\beta}(A)$. Esto da directamente una expresi\'{o}n de
$a$ como \eqref{noiny}.
\end{proof}

\begin{theorem}\label{bases}
Sea $B$ una extensi\'{o}n $(\pcomp,\scomp)$--acotada de $A$. Entonces
${}_AB$ es libre con base $\{ \mono{x}{\alpha} ~|~ \alpha \in
\BbbN^s \}$ si, y s\'{o}lo si, $f : \graded{F}{A} \rightarrow
\graded{\filext{F}}{B}$ es inyectivo y
${}_{\graded{F}{A}}\graded{\filext{F}}{B}$ es libre con base $\{
\mono{y}{\alpha} ~|~ \alpha \in \BbbN^s \}$.
\end{theorem}
\begin{proof}
Supongamos que $B$ es un $A$--m\'{o}dulo por la izquierda libre con
base $\{ \mono{x}{\alpha} ~|~ \alpha \in \BbbN^s \}$. Si $f$ no
fuese inyectivo entonces, por el Lema \ref{noinyl}, se tendr\'{\i}a una
expresi\'{o}n como \eqref{noiny} para alg\'{u}n $a \in A$. En este caso,
la expresi\'{o}n ha de ser, necesariamente, de la forma $a =
\sum_{\pcomp(\beta_1) \prec \pcomp(\mdeg{F}{a})}a^{\beta_1}$ de
donde deducimos que $a = \sum_{\beta_1 \prec_1
\mdeg{F}{a}}a^{\beta_1} \in \blackhole{F}{\mdeg{F}{a}}{A}$, lo que
es una contradicci\'{o}n. Por tanto, $f$ es inyectiva. Por el Teorema
\ref{relacionesgr}, $\{ \mono{y}{\alpha} ~|~ \alpha \in \BbbN^s
\}$ es un conjunto de generadores homog\'{e}neos de
${}_{\graded{F}{A}}\graded{\filext{F}}{B}$. Para demostrar que es
linealmente independiente, es suficiente, en vista del Teorema
\ref{relacionesgr}, con demostrar que en toda expresi\'{o}n de la
forma
\begin{equation}\label{deplin1}
\sum_{\peso{\alpha}= \alpha}r^{\alpha_1}\mono{y}{\alpha_2} = 0
\qquad (r^{\alpha_1} \in G^F_{\alpha_1}(A))
\end{equation}
se tiene necesariamente $r^{\alpha_1} = 0$ para todo $\alpha_1$.
Escribiendo $r^{\alpha_1} = a^{\alpha_1} +
\blackhole{F}{\alpha_1}{A}$ para cada $\alpha_1$, la ecuaci\'{o}n
\eqref{deplin1} se escribe, en vista de \eqref{monogr}, como
$\sum_{\peso{\alpha} = \alpha}a^{\alpha_1}\mono{x}{\alpha_2} \in
\blackhole{\filext{F}}{\alpha}{B}$. Por tanto, existe $\beta \prec
\alpha$ tal que
\begin{equation}\label{deplin2}
\sum_{\peso{\alpha} = \alpha}a^{\alpha_1}\mono{x}{\alpha_2} =
\sum_{\peso{\beta} \preceq \beta}b^{\beta_1}\mono{x}{\beta_2}
\end{equation}
Como $\{ \mono{x}{\gamma} ~|~ \gamma \in \BbbN^s \}$ es una base
de ${}_AB$, la igualdad \eqref{deplin2} es solo posible si para
cada sumando no nulo $a^{\alpha_1}\mono{x}{\alpha_2}$ del miembro
de la izquierda existe $(\beta_1, \beta_2)$ con $\peso{\beta}
\preceq \beta$ tal que $\alpha_2 = \beta_2$ y $a^{\alpha_1} =
b^{\beta_1}$. Como
\[
\pcomp(\beta_1) + \scomp(\alpha_2) = \pcomp(\beta_1) +
\scomp(\beta_2) \preceq \beta \prec \alpha = \peso{\alpha}
\]
deducimos que $\pcomp(\beta_1) \prec \pcomp(\alpha_1)$ y, por
tanto, se tiene que $\beta_1 \prec_1 \alpha_1$. Por tanto,
$a^{\beta_1} = b^{\beta_1} \in F_{\beta_1}(A) \subseteq
\blackhole{F}{\alpha_1}{A}$ para todo $\alpha_1$ y, de aqu\'{\i},
$r^{\alpha_1} = a^{\alpha_1} + \blackhole{F}{\alpha_1}{A} = 0$
para todo $\alpha_1$. Esto demuestra que los $\mono{y}{\gamma}$
forman una base de ${}_{\graded{F}{A}}\graded{\filext{F}}{B}$.

Supongamos rec\'{\i}procamente que $f$ es una aplicaci\'{o}n inyectiva y
que $\{ \mono{y}{\gamma} ~|~ \gamma \in \BbbN^s \}$ es una base de
${}_{\graded{F}{A}}\graded{\filext{F}}{B}$. La Proposici\'{o}n
\ref{esfiltracion} implica que $\{ \mono{x}{\gamma} ~|~ \gamma \in
\BbbN^s \}$ es un sistema de generadores de $B$ en tanto que
$A$--m\'{o}dulo por la izquierda. Si no es una base, entonces existe
una expresi\'{o}n del tipo
\begin{equation}\label{deplin3}
\sum_{(\alpha_1, \alpha_2) \in F}r^{\alpha_1}\mono{x}{\alpha_2} =
0
\end{equation}
para alg\'{u}n subconjunto finito $F \subseteq \BbbN^m \times
\BbbN^s$, donde $r^{\alpha_1} \in F_{\alpha_1}(A) \setminus
\blackhole{F}{\alpha_1}{A}$. Sea $\alpha$ el m\'{a}ximo, con respecto
de $\preceq$, de $\peso{\alpha}$ cuando $(\alpha_1, \alpha_2) \in
F$. Entonces, en $\graded{\filext{F}}{B}$, \eqref{deplin3} da
\begin{equation}\label{deplin4}
\sum_{\peso{\alpha} = \alpha}(r^{\alpha_1} +
\blackhole{\filext{F}}{\pcomp(\alpha_1)}{B})\mono{y}{\alpha_2} = 0
\end{equation}
Como $\{ \mono{y}{\gamma} ~|~ \gamma \in \BbbN^s \}$ es una base
de $\graded{\filext{F}}{B}$ visto como $\graded{F}{A}$--m\'{o}dulo por
la izquierda, tenemos que $r^{\alpha_1} \in
\blackhole{\filext{F}}{\alpha_1}{B}$ para todos los $\alpha_1$
implicados en \eqref{deplin4}. De esta forma, $r^{\alpha_1} =
\sum_{\peso{\beta} \preceq \beta \prec
\alpha}a^{\beta_1}\mono{x}{\beta_2}$ lo que implica, por el Lema
\ref{noinyl}, que $f$ no es una aplicaci\'{o}n inyectiva. Esto
concluye la demostraci\'{o}n del teorema.
\end{proof}

\begin{corollary}\label{basesOre}
Supongamos $A \subseteq B$ una extensi\'{o}n
$(\pcomp,\scomp)$--acotada tal que $\{ \mono{x}{\alpha} ~|~ \alpha
\in \BbbN^s \}$ es una base de ${}_AB$. Entonces
$\graded{\filext{F}}{B}$ es una extensi\'{o}n iterada de Ore
\[
\graded{\filext{F}}{B} =
\graded{F}{A}[y_1;\sigma_1]\cdots[y_s;\sigma_s]
\]
donde los endomorfismos $\sigma_i$ $(i = 1, \dots, s)$ verifican
que $\sigma_i(G^F_{\gamma}(A)) \subseteq G^F_{\gamma}(A)$ para
todo $\gamma \in \BbbN^m$ y $\sigma_j(y_i) = q_{ji}y_i$ para $1
\leqslant i < j \leqslant s$.
\end{corollary}
\begin{proof}
Es f\'{a}cil demostrarlo por inducci\'{o}n, teniendo en cuenta el Teorema
\ref{relacionesgr} y \cite[2.1.(iii)]{Goodearl/Letzter:1994}.
\end{proof}

\section{Un Teorema de Re-filtraci\'{o}n y regularidad de los grupos cu\'{a}nticos.}

Nos disponemos ahora a demostrar un teorema que permite, para
anillos multi-filtrados con multi-graduados asociados adecuados,
encontrar una nueva filtraci\'{o}n conservando el mismo anillo
graduado asociado.

\begin{theorem}\label{eneauno}
Sea $R$ una $K$--\'{a}lgebra dotada de una $(\BbbN^n,
\preceq)$--filtraci\'{o}n
\[
F = \{ F_{\alpha}(R) ~|~ \alpha \in \BbbN^n \}, \] donde $\preceq$
es un orden admisible cualquiera sobre $\BbbN^n$. Supongamos que
el \'{a}lgebra $\BbbN^n$--graduada asociada es una extensi\'{o}n iterada
de Ore
\[
\graded{F}{R} = \domain [y_1; \sigma_1] \dots [y_s; \sigma_s]
\]
donde $y_1,\dots, y_s$ son elementos homog\'{e}neos. Supongamos,
adem\'{a}s, que
\begin{enumerate}[(a)]
\item $\domain = F_0(R)$ es noetheriano por la izquierda;
\item para cada $1 \leqslant i < j \leqslant s$ existe $q_{ji} \in
\domain$ tal que $y_jy_i = q_{ji}y_iy_j$;
\item $F_{\alpha}(R)$ es un $\domain$--m\'{o}dulo por la izquierda
finitamente generado para cada $\alpha \in \BbbN^n$.
\end{enumerate}
Entonces $R$ puede ser dotada de una $\BbbN$--filtraci\'{o}n $\{ R_n
~|~ n \in \BbbN \}$ con $R_n$ verificando
\begin{enumerate}
\item $R_0 = \domain$;
\item $R_n$ es un $\domain$--m\'{o}dulo por la izquierda finitamente
generado para todo $n \in \BbbN$;
\item $gr(R) \cong \domain [y_1; \sigma_1] \dots [y_s; \sigma_s]$.
\end{enumerate}
\end{theorem}
\begin{proof}
Para cada $i = 1, \dots, s$, denotemos por $\alpha_i \in \BbbN^n$
el multi-grado de $y_i$. Es claro el conjunto $\{\mono{y}{\gamma}
~|~ \gamma \in \BbbN^s \}$ es una base de $\graded{F}{R}$ como
$\domain$--m\'{o}dulo por la izquierda. As\'{\i}, dado $r \in R$, el
elemento homog\'{e}neo $r + \blackhole{F}{\mdeg{}{r}}{R} \in
\graded{F}{R}$ tiene una representaci\'{o}n como polinomio est\'{a}ndar
homog\'{e}neo en los elementos homog\'{e}neos $y_1, \dots, y_s$ con
coeficientes por la izquierda en $\domain$. As\'{\i},
\begin{equation}\label{homog}
r + \blackhole{F}{\mdeg{}{r}}{R} = \sum_{\gamma_1\alpha_1 + \cdots
+ \gamma_s\alpha_s = \mdeg{}{r}}c_{\gamma}\mono{y}{\gamma},
\end{equation}
donde los $c_{\gamma}$ pertenecen a $\domain$. Tomemos, para cada
$i = 1, \dots, s$, un elemento $x_i \in F_{\alpha_i}(R)$ tal que
$y _i = x_i + \blackhole{F}{\alpha_i}{R}$ y denotemos por $M$ la
matriz de tama\~{n}o $s \times n$ cuyas filas son $\alpha_1, \dots,
\alpha_s$. La ecuaci\'{o}n \eqref{homog} se escribe
\begin{equation}\label{neweq1}
r + \blackhole{F}{\mdeg{}{r}}{R} = \sum_{\gamma M =
\mdeg{}{r}}c_{\gamma}\mono{x}{\gamma} +
\blackhole{F}{\mdeg{}{r}}{R}
\end{equation}
Como consecuencia, podemos demostrar por inducci\'{o}n sobre
$\mdeg{}{r}$ que
\begin{equation}\label{neweq2}
r = \sum_{\gamma M \preceq \mdeg{}{r}}a_{\gamma}\mono{x}{\gamma},
\end{equation}
donde $a_{\gamma} \in \domain$. Para deducir que $\{
\mono{x}{\gamma} ~|~ \gamma \in \BbbN^s \}$ es una base de $R$
como $\domain$--m\'{o}dulo por la izquierda, s\'{o}lo hemos de comprobar
la independencia lineal. Supongamos una relaci\'{o}n
\begin{equation}\label{ld}
\sum_{\gamma M \preceq \alpha} a_{\gamma}\mono{x}{\gamma} = 0,
\end{equation}
y hagamos inducci\'{o}n sobre $\alpha$. La relaci\'{o}n \eqref{ld} se
puede escribir
\begin{equation}\label{ld2}
\sum_{\gamma M = \alpha} a_{\gamma}\mono{x}{\gamma} + \sum_{\gamma
M \prec \alpha} a_{\gamma}\mono{x}{\gamma} = 0
\end{equation}
lo que, en $\graded{F}{R}$ da
\[
\sum_{\gamma M = \alpha} a_{\gamma}\mono{y}{\gamma} = 0
\]
Como los monomios $\mono{y}{\gamma}$ son $\domain$--linealmente
independientes, tenemos que $a_{\gamma} = 0$ para $\gamma M =
\alpha$. El resto de los coeficientes son nulos por inducci\'{o}n en
vista de \eqref{ld2}.

Sean ahora $a \in \domain$ e $i \in \{1, \dots, s \}$. Como
$G_0^{F}(R) = F_0(R) = \domain$, resulta de $y_ia =
\sigma_i(a)y_i$ que $\sigma_i(a)$ es de grado $0$, es decir,
$\sigma_i(a) \in \domain$. Escribiremos $a^{(i)} = \sigma_i(a)$.
Entonces
\begin{equation}\label{conmui}
0 = y_ia - a^{(i)}y_i = (x_ia - a^{(i)}x_i) +
\blackhole{F}{\alpha_i}{R}
\end{equation}
Como $\domain$ es noetheriano por la izquierda y $F_{\alpha_i}(R)$
es $\domain$--finitamente generado, tenemos que
$\blackhole{F}{\alpha_i}{R}$ noetheriano como $\domain$--m\'{o}dulo
por la izquierda. Por tanto deducimos de \eqref{conmui}, en
concurrencia con \eqref{neweq2}, que
\begin{equation}\label{neweq3}
x_i a = a^{(i)}x_i + \sum_{\gamma  \in
\Gamma_i}a_{\gamma}\mono{x}{\gamma},
\end{equation}
para ciertos $a_{\gamma} \in \domain$, donde $\Gamma_i$ es un
subconjunto finito de $\BbbN^s$ tal que $\gamma M \prec \alpha_i$
para todo $\gamma \in \Gamma_i$. Por otro lado, para $1 \leqslant
i < j \leqslant s$, tenemos
\begin{multline*}
0 = y_jy_i - q´_{ji}y_iy_j \\ = (x_j +
\blackhole{F}{\alpha_j}{R})(x_i + \blackhole{F}{\alpha_i}{R}) -
q_{ji}(x_i + \blackhole{F}{\alpha_i}{R})(x_j + \blackhole{F}{\alpha_j}{R}) \\
= (x_jx_i - q_{ji}x_ix_j) + \blackhole{F}{\alpha_i + \alpha_j}{R},
\end{multline*}
lo que, en vista de \eqref{neweq2}, nos da
\begin{equation}\label{neweq4}
x_jx_i - q_{ji}x_ix_j = \sum_{\gamma  \in
\Gamma_{ij}}a_{\gamma}\mono{x}{\gamma},
\end{equation}
donde $\Gamma_{ij}$ es un subconjunto finito de $\BbbN^s$ tal que
$\gamma M \prec \alpha_i + \alpha_j$ para todo $\gamma \in
\Gamma_{ij}$. Sea $\preceq'$ el orden admisible sobre $\BbbN^s$
definido por
\begin{equation}\label{ordenM}
\gamma \preceq' \mu \iff \begin{cases}
  \gamma M \prec \mu M & \text{o} \\
  \gamma M = \mu M \quad \text{y} \quad \gamma \leq_{\textrm{lex}} \mu
                        \end{cases}
\end{equation}
Dado que $\alpha_i = \epsilon_iM$ para todo $i = 1, \dots, s$, las
relaciones \eqref{neweq3} y \eqref{neweq4} se escriben
\begin{equation}\label{cadai}
x_ia-a^{(i)}x_i = \sum_{\gamma \prec' \epsilon_i \atop \gamma \in
\Gamma_i }a_{\gamma} \mono{x}{\gamma}
\end{equation}
y
\begin{equation}\label{cadaij}
x_jx_i - q_{ji}x_ix_j = \sum_{\gamma \prec' \epsilon_i +
\epsilon_j \atop \gamma \in \Gamma_{ij}}a_{\gamma}\mono{x}{\gamma}
\end{equation}
Sea $C = \{0\} \cup \big( \bigcup_{1 \leq i \leq s} C_i \big) \cup
\big( \bigcup_{1 \leq i < j \leq s} C_{ij} \big)$, donde $C_i =
\Gamma_i - \epsilon_i$ y $C_{ij} = \Gamma_{ij} - \epsilon_i -
\epsilon_j$. Es claro que $C$ es un subconjunto finito de
$\BbbZ^s$ cuyo m\'{a}ximo con respecto de $\preceq'$ es $0$. Por
\cite[Corollary 2.2]{Bueso/Gomez/Lobillo:1999unp}, existe
$\mathbf{w} = (w_1, \dots, w_s) \in \BbbN^n_+$ tal que
$\esc{\mathbf{w}}{\alpha} < 0$ para todo $\alpha \in C$. Esto
implica que las relaciones \eqref{cadai} y \eqref{cadaij} se
pueden escribir como
\begin{equation}\label{cadaiw}
x_ia - a^{(i)}x_i = \sum_{\esc{\mathbf{w}}{\gamma} <
w_i}a_{\gamma}\mono{x}{\gamma}
\end{equation}
y
\begin{equation}\label{cadaijw}
x_jx_i - q_{ji}x_ix_j = \sum_{\esc{\mathbf{w}}{\gamma} < w_i +
w_j}a_{\gamma}\mono{x}{\gamma}
\end{equation}
Por la Proposici\'{o}n \ref{esfiltracion}, tomando $\pcomp = 0$ y $
\scomp = \esc{\mathbf{w}}{-}: \BbbN^s \rightarrow \BbbN$,  podemos
dotar a $R$ de la filtraci\'{o}n $\{ R_n ~|~ n \in \BbbN \}$ dada por
\begin{equation}\label{filtro}
R_n = \sum_{\esc{\mathbf{w}}{\gamma} \leqslant
n}\Lambda\mono{x}{\gamma}
\end{equation}
Por el Corolario \ref{basesOre},
\[
gr(R) \cong \domain[y_1; \sigma_1]\cdots[y_s;\sigma_s]
\]
\end{proof}

\begin{corollary}\label{regular}
Supongamos que $R$ est\'{a} en las condiciones del Teorema
\ref{eneauno}. Supongamos, adem\'{a}s, que $q_{ji}$ es una unidad para
$1 \leqslant i < j \leqslant s$ y que $\sigma_i$ es un
automorfismo sobre $\domain$ para $i=1, \dots, s$. Si $\Lambda$ es
regular Auslander entonces $R$ es regular Auslander.
\end{corollary}
\begin{proof}
Si $\Lambda$ es regular Auslander, entonces, por \cite[Theorem
4.2]{Ekstrom:1989}, $gr(R) \cong \domain[y_1;\sigma_1] \cdots
[y_s;\sigma_s]$ es regular Auslander. El resultado se sigue ahora
de \cite[Theorem 3.9]{Bjork:1989}
\end{proof}

\begin{theorem}\label{main}
Sea $R$ un \'{a}lgebra sobre un cuerpo $\mathbf{k}$ dotada de una
$(\BbbN^n,\preceq)$--filtraci\'{o}n $\mathcal{F} = \{F_{\alpha}(R) ~|~
\alpha \in \BbbN^n \}$ verificando las hip\'{o}tesis del Teorema
\ref{eneauno}. Supongamos, adem\'{a}s, que
\begin{enumerate}
\item Los escalares $q_{ji}$ son unidades y los endomorfismos
$\sigma_i : \domain \rightarrow \domain$ son automorfismos.
\item $\domain$ est\'{a} generado por elementos $z_1, \dots, z_t$
tales que la filtraci\'{o}n est\'{a}ndar $\domain_n$ obtenida al asignar
grado $1$ a cada $z_i$ satisface que $gr(\domain)= \oplus_{n
\geqslant 0}\domain_n/\domain_{n-1}$ es un \'{a}lgebra finitamente
presentada y noetheriana.
\item $\sigma_i(\domain_1) \subseteq \domain_1$, para $i = 1, \dots, s$.
\item $gr(\domain)$ o bien
$\domain[y_1;\sigma_1]\cdots[y_s;\sigma_s]$ es un \'{a}lgebra regular
Auslander y Cohen-Macaulay.
\end{enumerate}
Entonces $R$ es un \'{a}lgebra regular Auslander y Cohen-Macaulay.
\end{theorem}
\begin{proof}
Sea $R_n$ la filtraci\'{o}n sobre $R$ proporcionada por el Teorema
\ref{eneauno} con $gr(R) = \domain
[y_1;\sigma_1]\cdots[y_s;\sigma_s]$. Como $\sigma_i(\domain_1)
\subseteq \domain_1$ para todo $i = 1, \dots, s$ y la filtraci\'{o}n
sobre $\domain$ es est\'{a}ndar, obtenemos que $y_i\domain_n \subseteq
\domain_ny_i$ para todo $i = 1, \dots, s$ y todo $n \geqslant 0$.
Por tanto, $\domain \subseteq \domain[y_1;\sigma_1] \cdots
[y_s;\sigma_s]$ es una extensi\'{o}n $(\pcomp, \scomp)$--acotada,
donde $\pcomp : \BbbN \rightarrow \BbbN$ es la identidad y $\scomp
: \BbbN^s \rightarrow \BbbN$ est\'{a} dado por $\scomp (\alpha) =
\esc{\mathbf{w}}{\alpha}$, para $\mathbf{w} = (w_1, \dots, w_s)$
con $w_i = \deg(y_i)$, $i= 1, \dots, s$. Por la Proposici\'{o}n
\ref{esfiltracion}, tenemos una filtraci\'{o}n sobre $gr(R) =
\domain[y_1,\sigma_1]\cdots[y_s;\sigma_s]$ dada por $gr(R)_{(n)} =
\sum_{i+\esc{\mathbf{w}}{\alpha}\leqslant
n}\domain_i\mono{y}{\alpha}$, cuyo graduado asociado es, por el
Corolario \ref{basesOre}, $gr(gr(R)) \cong
gr(\domain)[y_1;\sigma_1]\cdots[y_s;\sigma_s]$. Aqu\'{\i}, $\sigma_i$
denota el automorfismo graduado inducido en $gr(\domain)$ por el
hom\'{o}nimo automorfismo filtrado de $\domain$. Como $gr(\domain)$ es
un \'{a}lgebra noetheriana y finitamente presentada, deducimos que
$gr(gr(R))$ conserva estas mismas propiedades. Por tanto, la
filtraci\'{o}n $R_n$ est\'{a} en las hip\'{o}tesis de \cite[Theorem
1.3]{McConnell/Stafford:1989}.  Dado un $R$--m\'{o}dulo por la
izquierda $M$ finitamente generado, lo dotamos de una filtraci\'{o}n
tal que $gr(M)$ sea finitamente generado. Por \cite[Theorem
1.3]{McConnell/Stafford:1989}, $\GKdim (M) = \GKdim (gr(M))$. En
particular, $\GKdim (R) = \GKdim (gr(R))$. Por otra parte, de la
demostraci\'{o}n de \cite[Theorem 3.9]{Bjork:1989} obtenemos que
$j_R(M) = j_{gr(R)}(gr(M))$. Si suponemos que $gr(R) = \domain
[y_1;\sigma_1]\cdots[y_s;\sigma_s]$ es Cohen-Macaulay, obtenemos
\begin{multline*}
\mathrm{GKdim}(R) = \mathrm{GKdim} (gr(R)) = \\ j_{gr(R)}(gr(M)) +
\GKdim (gr(M)) = j_R(M) + \GKdim (M),
\end{multline*}
con lo que $R$ resulta ser as\'{\i}mismo Cohen-Macaulay.

Por \'{u}ltimo, si suponemos que $gr(\domain)$ es Cohen-Macaulay,
entonces $gr(gr(R))$ est\'{a} en las hip\'{o}tesis de
\cite[Lemma]{Levasseur/Stafford:1993}, por lo que es
Cohen-Macaulay. Dado que la filtraci\'{o}n $gr(R)_{(n)}$ es
finito-dimensional, obtenemos que $gr(R)$ es Cohen-Macaulay, con
lo que el razonamiento anterior demuestra que $R$ es
Cohen-Macaulay.
\end{proof}

\begin{theorem}
El \'{a}lgebra envolvente cuantizada $U_q(C)$ sobre $\mathbb{C}(q)$ de
una matriz de Cartan $C$ es regular Auslander y Cohen-Macaulay.
\end{theorem}
\begin{proof}
Seg\'{u}n \cite[Proposition 10.1]{DeConcini/Procesi:1993}, es posible
dotar a $U = U_q(C)$ de una $(\mathbb{N}^{n},\preceq)$--filtraci\'{o}n
$\{ F_{\alpha}(U) ~|~ \alpha \in \BbbN^n \}$ para cierto $n$ y
$\preceq$ un orden lexicogr\'{a}fico tal que el \'{a}lgebra multi-graduada
asociada $\graded{F}{U}$ es un espacio af\'{\i}n cu\'{a}ntico
multi-param\'{e}trico sobre $\mathbb{C}(q)$ con generadores
semi-conmutativos $z_1, \dots, z_t, y_1, \dots, y_s$ con $z_1,
\dots, z_t$ invertidos. Por la Proposici\'{o}n \ref{grCM},
$\graded{F}{U}$ es regular Auslander y Cohen-Macaulay. Adem\'{a}s,
$F_0(U) = \mathbb{C}(q)[z_1^{\pm 1},\dots,z_t^{\pm 1}]$, un anillo
de polinomios conmutativos de Laurent. Si dotamos a $F_0(U)$ de la
filtraci\'{o}n est\'{a}ndar obtenida al dar grado $1$ a $z_i^{\pm1}$ ($i =
1, \dots, t$), entonces $gr(F_0(U))$ es un factor del \'{a}lgebra de
polinomios en $2t$ variables con coeficientes en $\mathbb{C}(q)$.
Por tanto, es finitamente presentada y noetheriana. As\'{\i}, estamos
en las hip\'{o}tesis del Teorema \ref{main}, por lo que $U_q(C)$
resulta ser regular Auslander y Cohen-Macaulay.
\end{proof}

\providecommand{\bysame}{\leavevmode\hbox to3em{\hrulefill}\thinspace}
\providecommand{\MR}{\relax\ifhmode\unskip\space\fi MR }
\providecommand{\MRhref}[2]{%
  \href{http://www.ams.org/mathscinet-getitem?mr=#1}{#2}
}
\providecommand{\href}[2]{#2}

\end{document}